\documentclass[12pt,reqno]{amsart}
\usepackage{lmodern}
\usepackage{amsmath,amsthm,amscd,amsfonts,amssymb,graphicx,color}
\usepackage[bookmarksnumbered,colorlinks,plainpages]{hyperref}
\hypersetup{colorlinks=true,linkcolor=red,anchorcolor=green,citecolor=cyan,urlcolor=red,filecolor=magenta,pdftoolbar=true}
\DeclareMathOperator{\LIM}{LIM}
\usepackage{breqn}
\setkeys{breqn}{compact}
\usepackage[utf8]{inputenc}
\usepackage[T1]{fontenc}
\usepackage[nobreak]{cite}
\usepackage{mdwlist}
\usepackage{microtype}
\usepackage{mathtools}
\usepackage{enumitem}
\DeclareMathOperator{\card}{card}
\usepackage[capitalize,nameinlink]{cleveref}
\Crefname{enumi}{}{}
\Crefname{subsection}{Subsection}{Subsections}

\newtheorem{theorem}{Theorem}[section]
\newtheorem{lemma}[theorem]{Lemma}

\theoremstyle{definition}

\newtheorem{example}[theorem]{Example}

\newtheorem{problem}[theorem]{Problem}
\theoremstyle{remark}
\newtheorem{remark}[theorem]{Remark}
\numberwithin{equation}{section}

\makeatletter
\def\section{\@startsection{section}{1}%
  \z@{.7\linespacing\@plus\linespacing}{.5\linespacing}%
  {\normalfont\bfseries\centering}}
\makeatother

\begin{document}

\setcounter{page}{1}

\title[Linear operators involved in functional equations]{Some class of linear operators involved in functional equations}

\author{Janusz Morawiec}

\author{Thomas Zürcher}

\email{janusz.morawiec@us.edu.pl \textnormal{(corresponding author);}\newline\hphantom{{{\itshape E-mail addresses}\/:}}\,thomas.zurcher@us.edu.pl}

\address{Institute of Mathematics, University of Silesia, Bankowa 14, PL-40-007 Katowice, Poland}

\date{15th of November 2018}
%

\subjclass[2010]{Primary 47A50; Secondary 26A24, 39B12, 47B38.}

\keywords{linear operators; approximate differentiability; Luzin's condition N; functional equations; integrable solutions}


\begin{abstract}
Fix $N\in\mathbb N$ and assume that for every $n\in\{1,\ldots, N\}$ the functions $f_n\colon[0,1]\to[0,1]$ and $g_n\colon[0,1]\to\mathbb R$ are Lebesgue measurable, $f_n$ is almost everywhere approximately differentiable with $|g_n(x)|<|f'_n(x)|$ for almost all $x\in [0,1]$,  there exists $K\in\mathbb N$ such that the set $\{x\in [0,1]:\card{f_n^{-1}(x)}>K\}$ is of Lebesgue measure zero, $f_n$ satisfy Luzin's condition~N, and the set $f_n^{-1}(A)$ is of Lebesgue measure zero for every set $A\subset\mathbb R$ of Lebesgue measure zero. We show that the formula $Ph=\sum_{n=1}^{N}g_n\!\cdot\!(h\circ f_n)$ defines a linear and continuous operator $P\colon L^1([0,1])\to L^1([0,1])$, and then we obtain results on the existence and uniqueness of solutions $\varphi\in L^1([0,1])$ of the equation $\varphi=P\varphi+g$ with a given $g\in L^1([0,1])$.
\end{abstract}
\maketitle

\section{Introduction}
Let $P_\phi\colon L^1([0,1])\to L^1([0,1])$ be the Frobenius--Perron operator corresponding to the $2$-adic transformation $\phi\colon[0,1]\to[0,1]$ given by $\phi(x)=2x (\hspace{-1.5ex}\mod 1)$. It is known (see e.g.\ \cite[Example~4.1.1]{LaMaBook}) that $P_\phi$ is of the form
\begin{equation}\label{P0}
P_\phi f(x)=\frac{1}{2}f\left(\frac{x}{2}\right)+\frac{1}{2}f\left(\frac{x+1}{2}\right).
\end{equation}
The following result by Kazimierz Nikodem about the operator $P_\phi$ is the main motivation to write this paper.
\begin{theorem}[{see \cite[Theorem~2]{N1991}}]\label{thmKN}
Assume $g\in L^1([0,1])$. Then the equation
\begin{equation}\label{e0}
\varphi=P_\phi\varphi+g
\end{equation}
has a solution in $L^1([0,1])$ if and only if the series $\sum_{k=0}^{\infty}P_\phi^{k}g$ is convergent in $L^1([0,1])$. Moreover, every solution $\varphi\in L^1([0,1])$ of equation \eqref{e0} is of the form $\varphi=\sum_{k=0}^{\infty}P_\phi^{k}g+c$, where $c$ is a real constant.
\end{theorem}

In \cite{MZa} \cref{thmKN} was extended to the class of Frobenius--Perron operators that correspond to exact transformations and generalized to the class of ergodic Markov operators. In this paper we are interested in integrable solutions of equation \eqref{e0} in the case where the operator $P_\phi$ is replaced by a much more general one, not necessary by another Frobenius--Perron or Markov operator. Namely, we consider the following equation
\begin{equation}\label{e}
\varphi=P\varphi+g
\end{equation}
assuming that the operator $P$, acting on the space $L^1([0,1])$, is of the form
\begin{equation}\label{P}
Ph=\sum_{n=1}^{N}g_n\!\cdot\!(h\circ f_n)
\end{equation}
with functions $f_1,\ldots, f_N\colon[0,1]\to[0,1]$ and $g_1,\ldots, g_N\colon[0,1]\to\mathbb R$ satisfying a few extra conditions. To formulate all the extra conditions we need some definitions that we introduce in the next section. Now we only note that the extra conditions will exclude that formula \eqref{P} defines a Frobenius--Perron operator, and even a Markov operator (for details on Frobenius--Perron and Markov operators the reader can consult e.g.\ \cite[Chapters 6.2, 6.3 and 6.4]{LaMaBook}).

In this paper we are mainly interested in the case where $N\geq 2$. However, in the case where $N=1$ all results proved in the next sections hold true, because each summand on the right-hand side of \eqref{P} can be written in the form
\begin{equation*}
  g_{n,1}\!\cdot\!(h\circ f_n)+g_{n,2}\!\cdot\!(h\circ f_n)
\end{equation*}
with functions $g_{n,1},g_{n,2}\colon[0,1]\to\mathbb R$ having the same properties as the function $g_n$ and such that $g_{n,1}(x)+g_{n,2}(x)=g_n(x)$ for every $x\in[0,1]$. Let us note that equation \eqref{e} with $P$ of the form \eqref{P} and $N=1$ was widely discussed in \cite{KCG1990}; for its integrable solutions see e.g.\ Section~4.7 in this book.

The operator $P$ defined on $L^1([0,1])$ by formula \eqref{P} is linear, but it is not clear under which assumptions $P$ is continuous and sends functions from $L^1([0,1])$ into itself. Before we introduce assumptions guaranteeing that $P$ is continuous and $P\big(L^1([0,1])\big)\subset L^1([0,1])$, we note that a weak assumption involving $P$~and~$g$ implies the existence of an integrable solution of equation \eqref{e}.

\begin{remark}\label{solution of e}
Assume that $P\colon L^1([0,1])\to L^1([0,1])$ is a linear and continuous operator, and let $g\in L^1([0,1])$. If the series $\sum_{k=0}^{\infty}P^kg$ is convergent in $L^1([0,1])$, then
\begin{equation}\label{formula}
\varphi_0:=\sum_{k=0}^{\infty}P^kg
\end{equation}
is a solution of equation \eqref{e} belonging to the space $L^1([0,1])$.
\end{remark}

\begin{proof}
Define $\varphi_0$ by \eqref{formula} and observe that $\varphi_0\in L^1([0,1])$ by the assumption concerning the convergence. Then, making use of the linearity and continuity of~$P$, we obtain
\begin{equation*}
P\varphi_0+g=P\left(\sum_{k=0}^{\infty}P^kg\right)+g=\sum_{k=0}^{\infty}P^{k+1}g+g=\sum_{k=0}^{\infty}P^kg=\varphi_0.
\end{equation*}
The proof is complete.
\end{proof}

Throughout this paper we use the symbol $\varphi_0$ to represent the function defined by \eqref{formula}, which will be called the {\it elementary solution} of equation \eqref{e} provided that the series $\sum_{k=0}^{\infty}P^kg$ converges in $L^1([0,1])$. Thus we are interested in assumptions guaranteeing the convergence of the series $\sum_{k=0}^{\infty}P^kg$ for any $g\in L^1([0,1])$, or equivalently, guaranteeing that the elementary solution of equation \eqref{e} exists.

One of the simplest sufficient condition for the existence of the elementary solution of equation~\eqref{e} gives the following observation.

\begin{remark}\label{sufficient}
Assume that $P\colon L^1([0,1])\to L^1([0,1])$, and let $g\in L^1([0,1])$. If $P^mg=0$ for some $m\in\mathbb N$, then the elementary solution $\varphi_0$ of equation \eqref{e} exists and
$\varphi_0=\sum_{k=0}^{m-1}P^kg$.
\end{remark}

\begin{proof}
It suffices to note that the linearity of $P$ implies $P^kg=0$ for every $k\in\mathbb N$ such that $k>m$.
\end{proof}

The next example shows that the elementary solution of equation \eqref{e} can fail to exist, whereas an integrable solution can exist. Before we give the example, let us write that all concrete examples in this paper will concern equation \eqref{e} with~$P$ of the form \eqref{P} and $N=2$, i.e.\ the equation of the form
\begin{equation}\label{e-example}
\varphi(x)=g_1(x)\varphi(f_1(x))+g_2(x)\varphi(f_2(x))+g(x).
\end{equation}

\begin{example}\label{elementary solution1}
Fix a constant $c\in\mathbb R$ and let $g_1(x)=g_2(x)=-\frac{1}{2}$ and $g(x)=c$ for every $x\in[0,1]$. Then equation \eqref{e-example} takes the form
\begin{equation}\label{e1}
\varphi(x)=-\frac{1}{2}\varphi(f_1(x))-\frac{1}{2}\varphi(f_2(x))+c.
\end{equation}
It is clear that the constant function $\varphi=\frac{c}{2}$ satisfies \eqref{e1}, and since it belongs to the space $L^1([0,1])$, we conclude that equation \eqref{e1} has a solution in  $L^1([0,1])$. However, the elementary solution of this equation exists if and only if $c=0$, because in the case of equation \eqref{e1}, formula \eqref{P} yields $P^kg=(-1)^kc$ for every $k\in\mathbb N$ and hence $\sum_{k=0}^mP^kg=\frac{1}{2}(1+(-1)^m)c$ for every $m\in\mathbb N$.
\end{example}

If $c\neq0$, then the sequence  $(\sum_{k=0}^mP^kg)_{m\in\mathbb N}$ from \cref{elementary solution1} is not convergent in $L^1([0,1])$, but it is bounded. Thus we can calculate its pointwise classical Banach limit, that is $\LIM_{m\to\infty}(\sum_{k=0}^mP^kg)(x)=\frac{c}{2}$ for every $x\in[0,1]$. This suggests taking in \cref{solution of e} the pointwise classical Banach limit of the sequence $(\sum_{k=0}^mP^kg)_{m\in\mathbb N}$ instead of its limit. However, the problem is that there is no guarantee that the pointwise classical Banach limit of a bounded sequence of measurable (throughout this paper measurable always means Lebesgue measurable) functions is a measurable function (see \cite[page 288]{FKLWeiss1996}).

We finish this section with a remark about regularity of the elementary solution of equation \eqref{e}; talking about regularity properties of a function $h\in L^1([0,1])$, we mean that there exists a representative of $h$ having that regularity.

\begin{remark}\label{regularity1}
Assume that $P\colon L^1([0,1])\to L^1([0,1])$ is linear and continuous, \mbox{$g\in L^1([0,1])$}, and the elementary solution $\varphi_0$ of equation \eqref{e} exists.
\begin{enumerate}[label={\rm (\roman*)}]
\item If all the functions $f_1,\ldots,f_N,g_1,\ldots,g_N$ and $g$ are increasing (decreasing), then $\varphi_0$ is increasing (decreasing); in particular $\varphi_0$ is differentiable almost everywhere.
\item If all the functions $f_1,\ldots,f_N,g_1,\ldots,g_N$ and $g$ are continuous and if
\begin{equation}\label{Weierstrass}
\sum_{k=0}^{\infty}\|P^k g\|_\infty<\infty,
\end{equation}
then $\varphi_0$ is continuous.
\end{enumerate}
\end{remark}

\begin{proof}
(i) If $h\in L^1([0,1])$ is increasing (decreasing), then $Ph$ is as well. This implies that $P^kg$ is increasing (decreasing) for every $k\in\mathbb N$, and so is $\varphi_0$.

(ii) If $h\in L^1([0,1])$ is continuous, then $Ph$ is as well. This implies that $P^kg$ is continuous for every $k\in\mathbb N$. Condition \eqref{Weierstrass} guarantees the continuity of $\varphi_0$.
\end{proof}


\section{Preliminaries}
Let $E\subset\mathbb R$ be a nonempty set, let $x_0\in E$, and let $h\colon E\to\mathbb R$ be a function. We say that a linear mapping $L\colon\mathbb R\to\mathbb R$ is an \emph{approximate differential} of~$h$ at~$x_0$ if for every $\varepsilon>0$ the set
\begin{equation*}
\left\{x\in E\setminus\{x_0\}: \frac{|h(x)-h(x_0)-L(x-x_0)|}{|x-x_0|}<\varepsilon\right\}
\end{equation*}
has $x_0$ as a density point (see \cite{W1951}, cf. \cite[Chapter IX]{S1964}). We say that $h$ is \emph{approximately differentiable} at $x_0$ if the approximate differential of $h$ at $x_0$ exists.

We begin with the following easy observation, which also shows that the approximate differential (if it exists) is uniquely determined.

\begin{lemma}\label{Approx diff equal ae}
Assume that $E\subset\mathbb R$ is a nonempty set and $h_1,h_2\colon E\to\mathbb R$ are functions such that $h_1=h_2$ almost everywhere. If $h_1$ is approximately differentiable almost everywhere, then $h_2$ is as well. Moreover, whenever $h_1$~or~$h_2$ is approximately differentiable at a point, the other function is as well, and the approximate derivatives agree at this point.
\end{lemma}

\begin{proof}
Assume that $h_1(x_0)=h_2(x_0)$ and the function~$h_1$ is approximately differentiable at~\mbox{$x_0\in E$}; this is true for almost all $x_0\in E$. Let $L$ be the approximate differential of $h_1$ at $x_0$. Then for every $\varepsilon>0$, the sets
\begin{align*}
\left\{x\in E\setminus\{x_0\}: \frac{|h_1(x)-h_1(x_0)-L(x-x_0)|}{|x-x_0|}<\varepsilon\right\}
\end{align*}
and
\begin{align*}
\left\{x\in E\setminus\{x_0\}: \frac{|h_2(x)-h_2(x_0)-L(x-x_0)|}{|x-x_0|}<\varepsilon\right\}
\end{align*}
have the same measure.
Hence $h_2$ is approximately differentiable at~$x_0$ with $L$ being its approximate differential at $x_0$.
\end{proof}

To simplify notation, we will denote the approximate differential of a function $h\colon E\to\mathbb R$ at~$x_0$ by $h'(x_0)$. Moreover, if a function $h\colon E\to\mathbb R$ is almost everywhere approximately differentiable, then as usual we denote by $h'$ the function $E\ni x\mapsto h'(x)$, adopting the convention that $h'(x)=0$ for every point $x\in E$ in which $h$ is not approximately differentiable.

Assume that $U\subset \mathbb{R}$ is an open set. We say that $f\colon U \to \mathbb{R}$ satisfies \emph{Luzin's condition~N} if it maps sets of measure zero to sets of measure zero. If $f\colon U \to\mathbb{R}$ and if $E\subset U$, then the function $N_f(\cdot,E)\colon\mathbb{R}\to\mathbb{N}\cup \{\infty\}$ defined by
\begin{equation*}
N_f(y,E)=\card(f^{-1}(y)\cap E)
\end{equation*}
is called the \emph{Banach indicatrix} of~$f$.

\begin{lemma}\label{approximate derivative is measurable}
Assume that $U\subset \mathbb{R}$ is an open set and let $h\colon U\to\mathbb{R}$ be a measurable function satisfying Luzin's condition~N and being approximately differentiable almost everywhere. Then the approximate derivative $h'\colon U\to\mathbb R$ is a measurable function.
\end{lemma}

\begin{proof}
According to \cite[Theorem~1]{Hajlasz} (see also \cite[Theorem~1]{W1951}), there exist sequences $(X_k)_{k\in\mathbb N}$ of subsets of $U$ and $(h_k)_{k\in\mathbb N}$ of functions in $C^1(\mathbb{R})$ such that $X_k\subset X_{k+1}$ and $h'=h_k'$ almost everywhere in $X_k$ for every $k\in \mathbb{N}$, and moreover, the set $E=U\setminus \bigcup_{k=1}^\infty X_k$ is of measure zero. For every $k\in\mathbb N$ put
\begin{equation*}
Y_k=\{x\in X_k:h'(x)=h_k'(x)\}
\end{equation*}
and note that the set $X_k\setminus Y_k$ is of measure zero. Then for every open set $O\subset \mathbb{R}$ we have
\begin{align*}
{h'}^{-1}(O)=\,&\{x\in E: h'(x)\in O\}\cup\bigcup_{k=1}^\infty\{x\in X_k\setminus Y_k: h'(x)\in O\}\\
&\cup\bigcup_{k=1}^\infty {\{x\in Y_k: h_k'(x)\in O\}}.
\end{align*}
The first two sets are measurable as subsets of sets of measure zero. Since all $Y_k$~and~$h_k'$ are measurable, the third set is also measurable.
\end{proof}

The following change of variable theorem from \cite{Hajlasz} will be a useful tool in this paper.

\begin{theorem}[see {\cite[Theorem~2]{Hajlasz}}]\label{Hajlasz}
Assume that $U\subset \mathbb{R}$ is an open set and let $f\colon U\to\mathbb{R}$ be a measurable function satisfying Luzin's condition~N and being almost everywhere approximately differentiable.
If $h\colon \mathbb{R}\to \mathbb{R}$ is a measurable function, then for every measurable set $E\subset U$ the following statements are true:
\begin{enumerate}[label={\rm (\roman*)}]
\item\label{Hajlasz Measurable} The functions $(h\circ f)|f'|$ and $h N_f(\cdot,E)$ are measurable;
\item\label{Hajlasz Formula} If $h\geq 0$, then
\begin{equation}\label{change}
\int_E (h\circ f)(x)|f'(x)|dx=\int_{\mathbb{R}} h(y)N_f(y,E)dy;
\end{equation}
\item\label{Hajlasz Integrability} If one of the functions $(h\circ f)|f'|$ and $h N_f(\cdot,E)$ is integrable $($for $(h\circ f)|f'|$ integrability is considered with respect to~$E$$)$, then so is the other and \eqref{change} holds.
\end{enumerate}
\end{theorem}


\section{Main results}
From now on we fix $N\in\mathbb N$, measurable functions $f_1,\ldots,f_N\colon [0,1]\to [0,1]$ and  $g_1,\ldots, g_N\colon [0,1]\to\mathbb R$, and a function $g\in L^1([0,1])$.

For all $m\in\{1,\ldots,N\}$ and pairwise different $n_1,\ldots,n_m\in\{1,\ldots,N\}$
we put
\begin{equation*}
A_{n_1,\ldots,n_m}=\bigcap_{i=1}^mf_{n_i}([0,1])
\end{equation*}
and denote its measure by $l_{n_1,\ldots,n_m}$. Then we put
\begin{equation*}
L=\max\{m\in\{1,\ldots,N\}:l_{n_1,\ldots,n_m}>0\hbox{ for some }n_1<n_2<\cdots<n_m\}
\end{equation*}
and observe that
\begin{eqnarray}\label{L}
\sum_{n=1}^{N}\int_{f_n([0,1])}|h(y)|dy\leq L\int_{[0,1]}|h(y)|dy
\quad\hbox{ for every }h\in L^1([0,1]).
\end{eqnarray}

To formulate our results, we need the following hypotheses.
\begin{enumerate}[label=(H$_{\arabic*}$)]
\item\label{general assumption} For every $n\in\{1,\ldots, N\}$ the function $f_n$ is almost everywhere approximately differentiable and satisfy Luzin's condition~N.
\item\label{vanishing cardinality} There exists $K\in\mathbb N$ such that for every $n\in\{1,\ldots, N\}$ the set $\{x\in [0,1]:\card{f_n^{-1}(x)}>K\}$ is of measure zero.
\item\label{inverse of Luzin} The set $f_n^{-1}(A)$ is of measure zero for all $n\in\{1,\ldots, N\}$ and sets $A\subset\mathbb R$ of measure zero.
\end{enumerate}

Throughout the paper, $\|\cdot\|$ denotes the classical norm in $L^1([0,1])$ and $\|\cdot\|_*$ denotes the standard norm in the space of all linear and continuous operators from $L^1([0,1])$ into itself.

\begin{lemma}\label{lemP}
Assume {\rm \cref{general assumption}}, {\rm \cref{vanishing cardinality}}, and {\rm \cref{inverse of Luzin}}. If there exists a real constant $C$ such that
\begin{equation}\label{C}
|g_n(x)|\leq \frac{C}{KL}|f_n'(x)|\quad\hbox{for all }n\in\{1,\ldots,N\}\hbox{ and almost all }x\in [0,1],
\end{equation}
then formula \eqref{P} defines a linear operator $P\colon L^1([0,1])\to L^1([0,1])$, which is continuous satisfying
\begin{equation}\label{normP}
\|P\|_*\leq C.
\end{equation}
\end{lemma}

\begin{proof}
First, note that $P$ does not depend on the representatives by \cref{inverse of Luzin}. Moreover, the linearity of $P$ is evident.
	
Fix $h\in L^1([0,1])$.
	
Our first aim is to prove that $Ph\in L^1([0,1])$. Obviously, it is enough to show that $g_n\!\cdot\!(h\circ f_n)\in L^1([0,1])$ for every $n\in\{1,\ldots,N\}$.
	
Fix $n\in\{1,\ldots, N\}$ and extend $f_n$ and $h$ to the whole real line~$\mathbb{R}$ by letting them being~$0$ outside of the interval $[0,1]$; we will denote both the extensions by the same symbols $f_n$ and $h$, respectively.
	
Assumption~\cref{general assumption} allows us to use \cref{Hajlasz}. Then, applying assertion~\cref{Hajlasz Measurable} of~\cref{Hajlasz}, we see that the function $hN_{f_n}(\cdot,[0,1])$ is measurable. This function is integrable by \cref{vanishing cardinality}. Thus, making use of assertion \cref{Hajlasz Integrability} of \cref{Hajlasz}, we conclude that the function $(h\circ f_n)|f_n'|$ is integrable; in particular, it is measurable. Using~\eqref{C}, we obtain
\begin{equation*}
g_n(x)(h\circ f_n)(x)=\begin{cases}
g_n(x)\frac{(h\circ f_n)(x)|f_n'(x)|}{|f_n'(x)|}, & \mbox{if $f_n'(x)\not=0$},  \\
0, &  \mbox{if $f_n'(x)=0$}
\end{cases}
\end{equation*}
for almost all $x\in[0,1]$. This equality jointly with \cref{approximate derivative is measurable} and the measurability of $g_n$ and $(h\circ f_n)|f_n'|$ implies that the function $g_n\!\cdot\! (h\circ f_n)$ is measurable. Finally, by~\eqref{C} and the integrability of $(h\circ f_n)|f_n'|$, we conclude that $g_n\!\cdot\! (h\circ f_n)\in L^1([0,1])$.
	
Our second aim, which will complete the proof, is to show that $\|P\|_*\leq C$.
	
Applying~\eqref{C}, we get
\begin{align*}
\|Ph\|&=\int_{[0,1]}\left|\sum_{n=1}^{N}g_n(x)(h\circ f_n)(x)\right|dx
\leq\sum_{n=1}^{N}\int_{[0,1]}|g_n(x)(h\circ f_n)(x)|dx\\
&\leq\frac{C}{KL}\sum_{n=1}^{N}\int_{[0,1]}|f_n'(x)(h\circ f_n)(x)|dx.
\end{align*}
Note that for all $n\in\{1,\ldots, N\}$ and $y\in\mathbb R$ we have
\begin{equation*}
0<N_{f_n}(y,[0,1])\iff y\in f_n([0,1]).
\end{equation*}
Now extending all the considered functions to the whole real line~$\mathbb{R}$ by letting them being~$0$ outside of the interval $[0,1]$ (as in the first part of the proof) and then applying formula \eqref{change} of \cref{Hajlasz} jointly with \cref{vanishing cardinality} and \eqref{L}, we obtain
\begin{align*}
\|Ph\|&\leq\frac{C}{KL}\sum_{n=1}^{N}\int_{\mathbb{R}}|h(y)|N_{f_n}(y,[0,1])dy
\leq\frac{C}{L}\sum_{n=1}^{N}\int_{f_n([0,1])}|h(y)|dy\\
&\leq{C\int_{[0,1]}|h(y)|dy}=C\|h\|.
\end{align*}
This implies \eqref{normP} and completes the proof.
\end{proof}

Our first result reads as follows.

\begin{theorem}\label{first theorem}
Assume {\rm \cref{general assumption}}, {\rm \cref{vanishing cardinality}}, and {\rm \cref{inverse of Luzin}}. If \eqref{C} holds with a real constant $C<1$, then the elementary solution $\varphi_0$ of equation \eqref{e} exists, and it is the unique solution of equation \eqref{e} in $L^1([0,1])$.
	
Moreover, for every $m\in\mathbb N$ we have
\begin{equation*}
\left\|\varphi_0-\sum_{k=0}^{m}P^k g\right\|\leq \frac{C^{m+1}}{1-C}\|g\|.
\end{equation*}
\end{theorem}

\begin{proof}
We base our proof on the Banach fixed point theorem. To do so, we define an operator $T\colon L^1([0,1])\to L^1([0,1])$ by putting
\begin{equation*}
Tf=Pf+g;
\end{equation*}
this operator is well defined by \cref{lemP}.
	
Fix $h_1,h_2\in L^1([0,1])$. Applying \eqref{normP}, we obtain
\begin{align*}
\|Th_1-Th_2\|&=\|Ph_1-Ph_2\|\leq\|P\|_*\|h_1-h_2\|\leq  C\|h_1-h_2\|.
\end{align*}
Since $C<1$, we see that $T$ is a contraction.
	
By the Banach fixed point theorem, there exists exactly one $\varphi\in L^1([0,1])$ with $T\varphi=\varphi$, i.e.\ there exists exactly one $\varphi\in L^1([0,1])$ solving \eqref{e}.
Moreover, given $h\in L^1([0,1])$ we know that the unique solution $\varphi\in L^1([0,1])$ of \eqref{e} can be written as $\varphi=\lim_{m\to \infty} T^mh$, in particular
\begin{equation*}
\varphi=\lim_{m\to \infty} T^mg.
\end{equation*}
	
We claim that for every $m\in \mathbb{N}_0$ we have
\begin{equation}\label{T}
T^m g=\sum_{k=0}^{m}P^k{g}.
\end{equation}
The case $m=0$ is clear. So, fix $m\in\mathbb N$ and assume that \eqref{T} holds. Then
\begin{align*}
T^{m+1}g&=T\left(\sum_{k=0}^{m}P^kg\right)={P\left(\sum_{k=0}^{m}P^kg\right)+g}=\sum_{k=0}^{m+1}P^kg.
\end{align*}
In consequence, passing with $m$ to infinity in~\eqref{T}, we get  $\varphi=\sum_{k=0}^{\infty}P^k{g}=\varphi_0$.
	
For the speed of convergence note that \eqref{formula}~and~\eqref{normP} imply
\begin{align*}
\left\|\varphi_0-\sum_{k=0}^{m}P^k g\right\|&\leq\sum_{k=m+1}^{\infty}\|P^{k}g\|
\leq\sum_{k=m+1}^{\infty}\|P\|_*^k\|g\|\leq\sum_{k=m+1}^{\infty}C^k\|g\|=\frac{C^{m+1}}{1-C}\|g\|
\end{align*}
for every $m\in\mathbb N$.
\end{proof}

The following example says that the constant $C<1$ in condition \eqref{C} in \cref{first theorem}, is optimal, i.e.\ condition \eqref{C} cannot be replaced by the following weaker one
\begin{equation}\label{C1}
|g_n(x)|<\frac{1}{KL}|f_n'(x)|\quad\hbox{for all }n\in\{1,\ldots,N\}\hbox{ and almost all }x\in [0,1].
\end{equation}

\begin{example}\label{elementary solution6}
Fix $b\in\mathbb R$ and let $f_1(x)=f_2(x)=x$, $g_1(x)=g_2(x)=\frac{x}{2}$ and $g(x)=b$ for every $x\in[0,1]$. Then equation \eqref{e-example} takes the form
\begin{equation}\label{ex}
\varphi(x)=\frac{x}{2}\varphi(x)+\frac{x}{2}\varphi(x)+b.
\end{equation}
In the considered case $L=2$, {\rm \cref{general assumption}} and {\rm \cref{inverse of Luzin}} hold, and {\rm \cref{vanishing cardinality}} is satisfied with $K=1$. Moreover, \eqref{C} does not hold with any constant $C<1$, but \eqref{C1} is satisfied.
An easy calculation shows that the unique function satisfying \eqref{ex} everywhere on~$[0,1)$ is of the form
\begin{equation*}
\varphi(x)=\frac{b}{1-x},
\end{equation*}
but $\varphi\not\in L^1([0,1])$ provided that $b\neq 0$.
\end{example}

Note that equation \eqref{ex} can be written in the form
\begin{equation*}
\varphi(x)=x\varphi(x)+b.
\end{equation*}
Now $L=1$, {\rm \cref{vanishing cardinality}} is satisfied with $K=1$, {\rm \cref{general assumption}}, {\rm \cref{inverse of Luzin}}, and \eqref{C1} hold. However, \eqref{C} still does not hold with any constant $C<1$. This shows that the adopted assumptions are well-suited to our requirements.

The next observation together with a short comment after it sheds light on the meaning of condition \eqref{C1}.

\begin{remark}
Assume {\rm \cref{general assumption}}, {\rm \cref{vanishing cardinality}} and {\rm \cref{inverse of Luzin}}. If \eqref{C1} holds, then
\begin{equation}\label{inequality}
\sum_{n=1}^N\int_{[0,1]}|g_n(x)|dx<1.
\end{equation}
\end{remark}

\begin{proof}
Extend all the functions $f_1,\ldots,f_N$ and $g_1,\ldots,g_N$ to the whole real line~$\mathbb{R}$ by letting them being~$0$ outside of the interval $[0,1]$. Then applying \eqref{C1}, formula~\eqref{change} of \cref{Hajlasz} with $h=1$, \cref{vanishing cardinality}, \eqref{L} with the same  arguments as in the proof of \cref{lemP} we obtain
\begin{align*}
\sum_{n=1}^N\int_{[0,1]}|g_n(x)|\, dx&<\frac{1}{KL}\sum_{n=1}^N\int_{[0,1]}|f_n'(x)|\, dx=
\frac{1}{KL}\sum_{n=1}^{N}\int_{\mathbb{R}}N_{f_n}(y,[0,1])\, dy\\
&\leq\frac{1}{L}\sum_{n=1}^{N}\int_{f_n([0,1])}1\, dy\leq\int_{[0,1]}1\, dy=1,
\end{align*}
which completes the proof.
\end{proof}

Note that if the elementary solution $\varphi_0$ of equation \eqref{e} exists, then
$\varphi_0+c$ satisfies \eqref{e} for any real constant $c$ if and only if $P1=1$, i.e.\ if and only if
\begin{equation}\label{eqality}
\sum_{n=1}^Ng_n(x)=1\quad\hbox{for almost all }x\in[0,1].
\end{equation}
Clearly, condition \eqref{inequality} says something a little different from the negation of condition \eqref{eqality}.

Our second result gives conditions guaranteeing the existence of $\varphi_0$.

\begin{theorem}\label{second theorem}
Assume {\rm \cref{general assumption}}, {\rm \cref{vanishing cardinality}}, {\rm \cref{inverse of Luzin}}, and \eqref{C1}, and let
\begin{equation}\label{Pkneq0}
\|P^kg\|>0\quad\hbox{for every }k\in\mathbb N.
\end{equation}
If there exist $m\in\{1,\ldots,N\}$ and $C<1$ such that
\begin{equation}\label{C2}
|g_m(x)|\leq \frac{C}{KL}|f_m'(x)|\quad\hbox{for almost all }x\in [0,1]
\end{equation}
and
\begin{equation}\label{C3}
\inf_{k\in\mathbb N}\frac{1}{\|P^kg\|}\int_{f_m([0,1])}|P^kg(x)|dx>0,
\end{equation}
then the elementary solution of equation \eqref{e} exists.
\end{theorem}
	
\begin{proof}
By \cref{lemP} $P\colon L^1([0,1]\to L^1([0,1])$ is linear and continuous.

Since $\int_{f_m([0,1])}|P^kg(x)|dx\leq\|P^kg\|$ for every $k\in\mathbb N$, it follows that
\begin{equation*}
\alpha:=1-\frac{1-C}{L}\inf_{k\in\mathbb N}\frac{1}{\|P^kg\|}\int_{f_m([0,1])}|P^kg(x)|dx\in(0,1).
\end{equation*}

Using \eqref{C2}, \eqref{C1}, formula \eqref{change} of \cref{Hajlasz}, \cref{vanishing cardinality}, \eqref{L} and \eqref{C3} with the same arguments as in the second part of the proof of \cref{lemP}, we obtain
\begin{align*}
\|P^kg\|&\leq\sum_{n=1}^{N}\int_{[0,1]}|g_{n}(x)(P^{k-1}g\circ f_n)(x)|dx
\leq\frac{C}{KL}\int_{[0,1]}|f_m'(x)(P^{k-1}g\circ f_m)(x)|dx\\
&\hspace{3ex}+\frac{1}{KL}\sum_{n\neq m}\int_{[0,1]}|f_n'(x)(P^{k-1}g\circ f_n)(x)|dx\\
&=\frac{C}{KL}\int_{\mathbb{R}}|P^{k-1}g(y)|N_{f_m}(y,[0,1])dy
+\frac{1}{KL}\!\sum_{n\neq m}\int_{\mathbb{R}}|P^{k-1}g(y)|N_{f_n}(y,[0,1])dy\\
&\leq\frac{C}{L}\int_{f_m([0,1])}|P^{k-1}g(y)|dy
+\frac{1}{L}\sum_{n\neq m}\int_{f_n([0,1])}|P^{k-1}g(y)|dy\\
&=\frac{C-1}{L}\int_{f_m([0,1])}|P^{k-1}g(y)|dy
+\frac{1}{L}\sum_{n=1}^{N}\int_{f_n([0,1])}|P^{k-1}g(y)|dy\\
&\leq (\alpha-1)\|P^{k-1}g\|+\int_{[0,1]}|P^{k-1}g(y)|dy=\alpha\|P^{k-1}g\|
\end{align*}
for every $k\in\mathbb N$. Hence
\begin{equation*}
\|P^kg\|\leq\alpha^k\|g\|
\end{equation*}
for every $k\in\mathbb N,$ which ensures that the series $\sum_{k=0}^{\infty}P^kg$ converges in $L^1([0,1])$ and proves that the elementary solution of equation \eqref{e} exists.
\end{proof}

It is clear, in view of \cref{sufficient}, that assumption \eqref{Pkneq0} in \cref{second theorem} is not restrictive. It is also visible that condition \eqref{C3} holds trivially in the case where $f_m([0,1])=[0,1]$; in such a situation the infimum in \eqref{C3} is equal to~$1$. Combining this fact with \cref{elementary solution6}, we see that condition \eqref{C2} cannot be omitted in~\cref{second theorem}. Unfortunately, we do not know if condition \eqref{C3} in~\cref{second theorem} is necessary, so we formulate the following question.
	
\begin{problem}
Can we omit condition \eqref{C3} in \cref{second theorem}?
\end{problem}

We end this section with an application of \cref{second theorem}; compare it with \cref{elementary solution6}.

\begin{example}\label{elementary solution7}
Fix $a\in(2,+\infty)$, $b\in\mathbb R$ and let $f_1(x)=f_2(x)=x$, $g_1(x)=\frac{x}{2}$, $g_2(x)=\frac{x}{a}$ and $g(x)=b$ for every $x\in[0,1]$. Then equation \eqref{e-example} takes the form
\begin{equation}\label{ex1}
\varphi(x)=\frac{x}{2}\varphi(x)+\frac{x}{a}\varphi(x)+b.
\end{equation}
In the considered case $L=2$, {\rm \cref{general assumption}}, {\rm \cref{inverse of Luzin}} and \eqref{C1} hold, and {\rm \cref{vanishing cardinality}} is satisfied with $K=1$. Moreover,  \eqref{C2} and \eqref{C3} are satisfied with $C=\frac{2}{a}$ and $m=2$, since $f_2([0,1])=[0,1]$ (a priori, it can happen that $P^k b=0$ for some $k\in \mathbb{N}_0$; however, we see later that this can only occur if $b=0$; anyway the conclusion of \cref{second theorem} also holds in this case). Thus \cref{second theorem} yields that the elementary solution of equation \eqref{ex1} exists. In fact, equation \eqref{ex1} has no other solution in~$L^1([0,1])$, because an easy calculation shows that the unique function satisfying~\eqref{ex1} everywhere on $[0,1]$ is of the form
\begin{equation*}
\varphi(x)=\frac{2ab}{2a-(2+a)x}.
\end{equation*}
\end{example}


\section{Continuous dependence of elementary solutions}
In this section we will show that in general there is no continuous dependence of the elementary solution of equation \eqref{e} neither on the function $g$, nor on the functions $g_1,\ldots,g_N$, nor on the functions $f_1,\ldots,f_N$.

Fix a real number $\varepsilon\in\left(0,\frac{1}{2}\right)$.

The fact that there is no continuous dependence of the elementary solution of equation \eqref{e} on the function $g$ follows from the following example.

\begin{example}\label{elementary solution3}
By \cref{thmKN} there exists the elementary solution of equation~\eqref{e0} with $g=0$; it is the trivial function. But \cref{thmKN} also says that the elementary solution of equation \eqref{e0} with $g=\varepsilon$ does not exist, because in this case $P1=1$, and hence $\sum_{k=0}^mP^k\varepsilon=(m+1)\varepsilon$ for every $m\in\mathbb N$.	
\end{example}	

The next example shows that there is no continuous dependence of the elementary solution of equation \eqref{e} on the functions $g_1,\ldots,g_N$.

\begin{example}\label{elementary solution4}
Let $f_1(x)=\frac{x}{2}$, $f_2(x)=\frac{x+1}{2}$, $g_1(x)=g_2(x)=\frac{1-\varepsilon}{2}$ and $g(x)=1$ for every $x\in[0,1]$. Then equation \eqref{e-example} takes the form	
\begin{equation}\label{evarepsilon}
\varphi(x)=\frac{1-\varepsilon}{2}\varphi\left(\frac{x}{2}\right)+\frac{1-\varepsilon}{2}\varphi\left(\frac{x+1}{2}\right)+1.
\end{equation}
In the considered case, $L=1$, \cref{general assumption} and \cref{inverse of Luzin} are satisfied, and \cref{vanishing cardinality} holds with $K=1$. Moreover, \eqref{C} holds with $C=1-\varepsilon$. Hence by \cref{first theorem} the elementary solution of equation \eqref{evarepsilon} exists, it is the unique integrable solution of equation \eqref{evarepsilon} and it has the form
\begin{equation*}
\varphi_{0,\varepsilon}(x)=\sum_{k=0}^{\infty}(1-\varepsilon)^k=\frac{1}{\varepsilon}\quad\hbox{for every }x\in[0,1].
\end{equation*}
According to \cref{thmKN} and repeating the same arguments as in \cref{elementary solution3}, we see that the elementary solution of equation \eqref{evarepsilon} with $\varepsilon=0$ does not exist.
\end{example}

The last example shows that also there is no continuous dependence of the elementary solution of equation \eqref{e} on the functions $f_1,\ldots,f_N$.

\begin{example}\label{elementary solution5}
Let $f_1(x)=x^{1+\varepsilon}$, $f_2(x)=x^{1+2\varepsilon}$, $g_1(x)=g_2(x)=\frac{x}{2}$ and $g(x)=1$ for every $x\in[0,1]$. Then equation \eqref{e-example} takes the form
\begin{equation}\label{e5}
\varphi(x)=\frac{x}{2}\varphi\left(x^{1+\varepsilon}\right)+\frac{x}{2}\varphi\left(x^{1+2\varepsilon}\right)+1.
\end{equation}
In the considered case, $L=2$, \cref{general assumption} and \cref{inverse of Luzin} are satisfied, and \cref{vanishing cardinality} holds with $K=1$. Moreover, \eqref{C} holds with $C=\frac{1}{1+\varepsilon}$. Hence by \cref{first theorem} the elementary solution of equation \eqref{e5} exists and it is the unique integrable solution of equation~\eqref{e5}. Taking into account
\cref{elementary solution6}, we see that the elementary solution of equation \eqref{e5} with $\varepsilon=0$ does not exist.
\end{example}


\section{A special case of equation \texorpdfstring{\eqref{e-example}}{(1.6)}}
Fix constants $a,b\in\mathbb R$ and let $f_1(x)=\frac{x}{2}$, $f_2(x)=\frac{x+1}{2}$, $g_1(x)=g_2(x)=a$ and $g(x)=b$ for every $x\in[0,1]$. Then equation \eqref{e-example} takes the form
\begin{equation}\label{e2}
\varphi(x)=a\varphi\left(\frac{x}{2}\right)+a\varphi\left(\frac{x+1}{2}\right)+b.
\end{equation}
In the considered case, $L=1$, \cref{general assumption} and \cref{inverse of Luzin} are satisfied, and \cref{vanishing cardinality} holds with $K=1$. Moreover, \eqref{C} holds with $C=2|a|$. Then \cref{lemP} implies that formula \eqref{P} defines a linear and continuous operator $P\colon L^1([0,1])\to L^1([0,1])$ and a simple calculation shows that
\begin{equation*}
\sum_{k=0}^mP^kb=b\sum_{k=0}^m(2a)^k.
\end{equation*}
Hence the elementary solution $\varphi_0$ of equation \eqref{e2} exists if and only if $b=0$ or if $|a|<\frac{1}{2}$, and moreover, $\varphi_0$ is a constant function and
\begin{equation*}
\varphi_0=
\begin{cases*}
\frac{b}{1-2a},&in the case where $|a|<\frac{1}{2}$,\\
0,&in the case where $b=0$.
\end{cases*}
\end{equation*}

If $|a|<\frac{1}{2}$, then \cref{first theorem} implies that the elementary solution of equation~\eqref{e2} is the unique solution of equation~\eqref{e2} in $L^1([0,1])$.

If $a=\frac{1}{2}$, then the elementary solution of equation \eqref{e2} exists if and only if $b=0$. Moreover, by \cref{thmKN} equation \eqref{e2} has no solution in $L^1([0,1])$ in the case where $b\neq 0$, and the elementary solution of equation \eqref{e2} is the unique, up to an additive constant, solution of equation \eqref{e2} in $L^1([0,1])$ in the case where $b=0$; actually if $b=0$, then only constant functions are integrable solutions of equation \eqref{e2}.

If $a=-\frac{1}{2}$, then equation \eqref{e2} is a special case of equation \eqref{e1}. From \cref{elementary solution1} we see that equation \eqref{e2} has a solution in $L^1([0,1])$. However, the elementary solution of equation \eqref{e2} exists only in the case where $b=0$.

If $|a|>\frac{1}{2}$, then (as in the previous case) the elementary solution of equation \eqref{e2} exists only in the case where $b=0$. However, equation \eqref{e2} has a solution in $L^1([0,1])$ in the case where $b\neq 0$; the constant function $\frac{b}{1-2a}$ satisfies \eqref{e2}, and it coincides with the above formula of $\varphi_0$, although the elementary solution does not exist in this case.

We know for which parameters $a,b\in\mathbb R$ the elementary solution of equation~\eqref{e2} exists. However, we do not know what the set of all integrable solutions of equation~\eqref{e2} looks like in the case where $a\in(-\infty,-\frac{1}{2}]\cup(\frac{1}{2},+\infty)$ and $b\in\mathbb R$. Let us only note here that for certain parameters $a,b\in\mathbb R$ the set of integrable solutions of equation \eqref{e2} can be very large (see e.g.\ \cite{MZ2018}, where a large class of continuous solutions $\varphi\colon[0,1]\to\mathbb R$ of equation \eqref{e2} with $a=1$ has been found; actually there $b$ was equal to $-\varphi\left(\frac{1}{2}\right)$, however, fixing a function $\varphi\colon[0,1]\to\mathbb R$ satisfying equation \eqref{e2} with $a=1$ and $b=-\varphi(\frac{1}{2})$ for every $x\in[0,1]$, and then defining $\psi\colon[0,1]\to\mathbb R$ by putting $\psi(x)=\varphi(x)-\varphi(\frac{1}{2})-b$, we deduce that $\psi(x)=\psi\left(\frac{x}{2}\right)+\psi\left(\frac{x+1}{2}\right)+b$ for every $x\in[0,1]$.)
We conclude our discussion on integrable solutions of equation~\eqref{e2} with the following problem.

\begin{problem}
Determine all integrable solutions of equation \eqref{e2} in the case where $a\in(-\infty,-\frac{1}{2}]\cup(\frac{1}{2},+\infty)$ and $b\in\mathbb R$.
\end{problem}

\bigskip

{\bf Acknowledgments.} This research was supported by the University of Silesia Mathematics Department (Iterative Functional Equations and Real Analysis program).

\bibliographystyle{amsplain}
\bibliography{GeneralizedDerivative}

\providecommand{\MR}{\relax\ifhmode\unskip\space\fi MR }
\providecommand{\MRhref}[2]{%
  \href{http://www.ams.org/mathscinet-getitem?mr=#1}{#2}
}
\providecommand{\href}[2]{#2}
\begin{thebibliography}{1}

\bibitem{Hajlasz}
P.~Haj{\l}asz, \emph{Change of variables formula under minimal assumptions},
  Colloq. Math. \textbf{64} (1993), no.~1, 93--101. \MR{1201446}

\bibitem{KCG1990}
M.~Kuczma, B.~Choczewski, and R.~Ger, \emph{Iterative functional equations},
  Encyclopedia of Mathematics and its Applications, vol.~32, Cambridge
  University Press, Cambridge, 1990. \MR{1067720}

\bibitem{LaMaBook}
Andrzej Lasota and Michael~C. Mackey, \emph{Chaos, fractals, and noise}, second
  ed., Applied Mathematical Sciences, vol.~97, Springer-Verlag, New York, 1994,
  Stochastic aspects of dynamics. \MR{1244104}

\bibitem{MZa}
J.~Morawiec and T.~Z\"{u}rcher, \emph{An application of functional equations
  for generating $\varepsilon$-invariant measures}, 2018, arXiv:1810.04530.

\bibitem{MZ2018}
Janusz Morawiec and Thomas Z\"{u}rcher, \emph{On a problem of {J}anusz
  {M}atkowski and {J}acek {W}eso\l owski}, Aequationes Math. \textbf{92}
  (2018), no.~4, 601--615. \MR{3831249}

\bibitem{N1991}
K.~Nikodem, \emph{On {$\epsilon$}-invariant measures and a functional
  equation}, Czechoslovak Math. J. \textbf{41(116)} (1991), no.~4, 565--569.
  \MR{1134949}

\bibitem{S1964}
S.~Saks, \emph{Theory of the integral}, Second revised edition. English
  translation by L. C. Young. With two additional notes by Stefan Banach, Dover
  Publications, Inc., New York, 1964. \MR{0167578}

\bibitem{FKLWeiss1996}
Benjamin Weiss, \emph{A survey of generic dynamics}, Descriptive set theory and
  dynamical systems ({M}arseille-{L}uminy, 1996), London Math. Soc. Lecture
  Note Ser., vol. 277, Cambridge Univ. Press, Cambridge, 2000, pp.~273--291.
  \MR{1774430}

\bibitem{W1951}
H.~Whitney, \emph{On totally differentiable and smooth functions}, Pacific J.
  Math. \textbf{1} (1951), 143--159. \MR{0043878}

\end{thebibliography}

\end{document}